\documentclass[a4paper,11pt,oneside]{article}
\usepackage[latin1]{inputenc}
\usepackage{amssymb, amsmath}
\usepackage{moreverb}
\usepackage{vmargin}
\usepackage{fancyhdr}
\usepackage{amsfonts}
\usepackage{fancybox}
\usepackage{amsthm}
\usepackage{amsmath}

\newtheorem{theorem}{Theorem}[section]

\newtheorem{lemma}[theorem]{Lemma}

\newtheorem{definition}[theorem]{Definition}

\usepackage[
  linktocpage,
  urlbordercolor={1 0.5 0.125},
  filebordercolor={1 0.5 0.125},
  linkbordercolor={0.25 1 0.25},
  citebordercolor={0.25 1 0.25},
  pdfborderstyle={/S/U/W 1}, 
  breaklinks]{hyperref}  
  \begin{document}
  \begin{center}
\begin{Large}
Global existence of solutions to a parabolic-elliptic chemotaxis system with critical degenerate diffusion \\
\end{Large}
\vspace{0.5cm}
Elissar Nasreddine\\
\vspace{0.2cm}
 \begin{small}
\textit{Institut de Math\'ematiques de Toulouse, Universit\'e de Toulouse,}\\
 \textit{F--31062 Toulouse cedex 9, France}\\
 \vspace{0.1cm}
 e-mail: elissar.nasreddine@math.univ-toulouse.fr\\
 \today
\end{small}
\end{center}

\textbf{Abstract} This paper is devoted to the analysis of non-negative solutions for a degenerate parabolic-elliptic Patlak-Keller-Segel system with critical nonlinear diffusion in a bounded domain with homogeneous Neumann boundary conditions. Our aim is to prove the existence of a global weak solution under a smallness condition on the mass of the initial data, there by completing previous results on finite blow-up for large masses. Under some higher regularity condition on solutions, the uniqueness of solutions is proved by using a classical duality technique.
\vspace{0.5cm}
\section{Introduction}
Chemotaxis is the movement of biological organisms oriented towards the gradient of some substance, called the chemoattractant. The Patlak-Keller-Segel (PKS) model (see \cite{initiation}, \cite{onexplo} and \cite{random}) has been introduced in order to explain chemotaxis cell aggregation by means of a coupled system of two equations: a drift-diffusion type equation for the cell density $u$, and a reaction diffusion equation for the chemoattractant concentration $\varphi$. It reads
\begin{equation}
\label{PKS}
(PKS)\left\{
\begin{array}{rclr}
\partial_t u&=&\mathrm{div}(\nabla u^m-u\cdot\nabla \varphi)& x\in \Omega, t>0,\\
-\Delta \varphi & =&u-<u> &x \in \Omega,  t>0,\\
<\varphi(t)>&=&0 & t>0,\\
\partial_\nu u=\partial_\nu \varphi&=&0 & x\in \partial \Omega, t>0,\\
u(0,x)&=&u_0(x)& x\in \Omega,
\end{array}
\right.
\end{equation}
where $\Omega  \subset \mathbb{R}^N$ is an open bounded domain, $\nu$ the outward unit normal vector to the boundary $\partial \Omega$ and $m\geq 1$. An important parameter in this model is the total mass $M$ of cells, which is formally conserved through the evolution:
 \begin{equation}
 M=<u>=\frac{1}{|\Omega|}\int _\Omega u(t, x)\ dx
 =\frac{1}{|\Omega|}\int _\Omega u_0(x)\ dx.
 \end{equation}
 Several studies have revealed that the dynamics of \eqref{PKS} depend sensitively on the parameters $N$, $m$ and $M$. More precisely, if $N=2$ and $m=1$, it is well-known that the solutions of \eqref{PKS} may blow up in finite time if $M$ is sufficiently large (see \cite{random, blow}) while solutions are global in time for $M$ sufficiently small \cite{random}, see also the survey articles \cite{onthepar, from}. \\
 
The situation is very different when $m=1$ and $N\ne 2$. In fact, if $N=1$, there is global existence of solutions of \eqref{PKS} whatever the value of the mass of initial data $M$, see \cite{finite} and the references therein. If $N\ge 3$, for all $M>0$, there are initial data $u_0$ with mass $M$ for which the corresponding solutions of \eqref{PKS} explode in finite time (see \cite{blow}). Thus, in dimension $N\ge 3$ and $m=1$, the threshold phenomenon does not take place as in dimension $2$, but we expect the same phenomenon when $N\ge 3$ and $m$ is equal to the \textit{critical} value $m=m_c=\frac{2(N-1)}{N}$. More precisely, we consider a more general version of \eqref{PKS} where the first equation of \eqref{PKS} is replaced by
 $$\partial_t u=\mathrm{div}(\phi(u)\ \nabla u-u\ \nabla \varphi), \quad t>0, \quad x\in\Omega,$$
and the diffusitivity $\phi$ is a positive function in $C^1([0,\infty[)$ which does not grow to fast at infinity. In \cite{finite}, the authors proved that there is a critical exponent such that, if the diffusion has a faster growth than the one given by this exponent, solutions to \eqref{PKS} (with $\phi(u)$ instead of $m u^{m-1}$) exist globally and are uniformly bounded, see also \cite{volume, preventing} for $N=2$. More precisely, the main results in \cite{finite} read as follows:
 \begin{itemize}
 \item If $\phi(u)\geq c(1+u)^p$ for all $u\geq 0$ and some $c>0$ and $p> 1-\frac{2}{N}$ then all solutions of \eqref {PKS} are global and bounded. 
 \item If  $\phi(u)\leq c(1+u)^p$ for all $u\geq 0$ and some $c>0$ and $p< 1-\frac{2}{N}$ then there exist initial data $u_0$ such that 
 $$\lim_{t \to T}||u(.,t)||_{\infty}=\infty, \ \mathrm{for\ some\ finite\ T>0}.$$
 \end{itemize}
Except for $N=2$, the critical case $m=\frac{2(N-1)}{N}$ is not covered by the analysis of \cite{finite}. Recently, Cie\'slak and Lauren\c{c}ot in \cite{finitetim} show that if $\phi(u)\leq c(1+u)^{1-\frac{2}{N}}$ and $N\ge 3$, there are solutions of \eqref{PKS} blowing up in finite time when $M$ exceeds an explicit threshold. In order to prove that, when $N \geq 3$ and $m =\frac{2\ (N - 1)}{N}$, we have a threshold phenomenon similar to dimension $N = 2$ with $m = 1$, it remains to show that solutions of \eqref{PKS} are global when $M$ is small enough. The goal of this paper is to show that this is indeed true, see Theorem \ref{Ex} below.\\

By combining Theorem \ref{Ex} with the blow-up result obtained in \cite{finitetim}, we conclude that, for $N\ge 3$ and $m=\frac{2(N-1)}{N}$, there exists $0<M_1\le M_2<\infty$ such that the solutions of \eqref{PKS} are global if the mass $M$ of the initial data $u_0$  is in $[0,M_1)$, and may explode in finite time if $M>M_2$. An important open question is whether $M_1=M_2$ when $\Omega$ is a ball in $\mathbb{R}^N$ and $u_0$ is a radially symmetric function. Notice that, in the radial case, this result is true when $ N = 2 $ and $ m = 1$, and the threshold value of the mass for blow-up is $ M_1 = M_2 = 8 \pi $, see \cite{volume, blow, blowup, pde}. Again, for $N=2$ and $m=1$, but for regular, connected and bounded domain, it has been shown that $M_1=4\pi=\frac{M_2}{2}$ (see \cite{blowup,blow} and the references therein). Such a result does not seem to be
known for $N\ge 3$ and $m=\frac{2(N-1)}{N}$.\\

 Still, in the whole space $\Omega=\mathbb{R}^N$ when the equation for $\varphi$ in \eqref{PKS} is replaced by the Poisson equation $\varphi = E_N*u$, with $E_N$ being the Poisson kernel, it has been shown in \cite{optimal, critical, local, global, time, globalexis} that:
 \begin{itemize} 
 \item When $N\ge 3$ and $1\le m< 2-\frac{2}{N}$, this modified version of \eqref{PKS} has a global weak solution if $M=\|u_0\|_1$ is sufficiently small, while finite time blow-up occurs for some initial data with sufficiently large mass.
\item When $N\ge 2$ and $m>2-\frac{2}{N}$, this modified version of \eqref{PKS} has a global weak solution whatever the value of $M$.
\item When $N\ge 2$ and $m=2-\frac{2}{N}$, there is a threshold mass $M_c>0$ such that solutions to this modified version of \eqref{PKS} exist globally if $M=\|u_0\|_1\le M_c$, and might blow up in finite time if $M>M_c$.
 \end{itemize}

From now on, we assume that
$$
N\ge 3 \quad\mbox{ and }\quad m=\frac{2(N-1)}{N} .
$$
\section{Main Theorem}
Throughout this paper , we deal with weak solutions of \eqref{PKS}. Our definition of weak solutions now reads:
\begin{definition}\label{def}
Let $T\in (0;\infty]$. A pair $(u, \varphi)$ of functions $u: \Omega\times [0,T)\longrightarrow [0,\infty)$, $\varphi: \Omega \times [0,T)\longrightarrow \mathbb{R}$ is called a weak solution of (\ref{PKS}) in $\Omega \times [0,T)$ if 
\begin{itemize}
\item $u\in L^\infty((0,T); L^\infty(\Omega))$; $\ u^m \in L^2((0,T); H^1(\Omega))$ and $<u>=M$.
\item $\varphi \in L^2((0,T); H^1(\Omega))$ and $<\varphi>=0$.

\item $(u,\varphi)$ satisfies the equation in the sense of distributions ; i.e,
$$-\int _0^T\int_\Omega \left(\nabla u^m\cdot \nabla \psi- u\nabla \varphi\cdot \nabla\psi-u\ \partial_t\psi\right)\ \mathrm{d}x\mathrm{d}t=\int_\Omega u_0(x)\ \psi(0, x)\ \mathrm{d}x,$$
$$\int ^T_0\int_\Omega \nabla\varphi\cdot \nabla \psi\ dxdt=\int_0^T\int_\Omega (u-M)\ \psi \ \mathrm{d}x\mathrm{d}t,
$$
\end{itemize}
for any continuously differentiable function $\psi \in C^1([0,T]\times \overline{\Omega})$ with $\psi(T)=0$ and $T>0$.

\end{definition}
For $\varphi\in H^1(\Omega)$ satisfying $<\varphi>=0$, we denote by $C_s$ the Sobolev constant where
\begin{equation}\label{sob}||\nabla \varphi||_2\geq C_s ||\varphi||_{2^*},\ \ \mathrm{where} \ 2^*=\frac{2N}{N-2}.\end{equation}
The main theorem gives the existence and uniqueness of a time global weak solution to \eqref{PKS} which corresponds to a degenerate version of the ``Nagai model" for the semi-linear Keller-Segel system, when $u_0\in L^\infty(\Omega)$ and the initial data is assumed to be small.
\begin{theorem}\label{Ex}
Define
\begin{equation}\label{mass}M_*:=\left(\frac{2\ C_s^{2}}{(m-1)\ |\Omega|^{\frac{2}{N}}}\right)^{\frac{N}{2}},\end{equation} 
where $C_s$ is the Sobolev constant in \eqref{sob}.\\
Assume that $u_0$ is nonnegative function in $L^\infty(\Omega)$, which satisfies
 \begin{equation}\label{cond}||u_0||_{1}<M_*.\end{equation} 
Then the equation (\ref{PKS}) has a global weak solution $(u,\varphi)$ in the sense of Definition \ref{def}. Moreover, if we assume that 
\begin{equation} \label{assu} \varphi\in L^\infty((0, T); W^{2, \infty}(\Omega))\end{equation} for all $T>0$
then this solution is unique.
\end{theorem}
In order to prove the previous theorem, we introduce the following approximated equations
\begin{equation*}
(KS)_\delta \left\{
\begin{array}{rclr}
\partial_t u_\delta&=&\mathrm{div}\left(\nabla (u_\delta+\delta)^m-u_\delta \nabla \varphi_\delta\right)& x\in \Omega, t>0,\\
-\Delta \varphi_\delta & =&u_\delta-<u_\delta>&x \in \Omega, t>0,\\
\partial_\nu u_\delta =\partial_\nu \varphi_\delta&=&0 & x\in \partial \Omega, t>0,\\
u_\delta(0, x)&=&u_0(x)& x \in \Omega,
\end{array}
\right.
\end{equation*}
where $\delta \in (0,1)$, and we show that under a smallness condition on the mass of initial data, the Liapunov function 
$$L_\delta(u,\varphi)=\int_\Omega (b_\delta(u)+\frac{1}{2}|\nabla\varphi_\delta|^2-u_\delta\ \varphi_\delta)\ \mathrm{d}x,$$
yields the $L^m$ bound of $u_\delta(t)$ independent of $\delta$. Then using Gagliardo-Nirenberg and Poincar\'e inequalities, we obtain for $p>m$, the $L^p$ bound for $u_\delta(t)$ independent of $\delta$. As a consequence of Sobolev embedding theorem, we improve the regularity of $\varphi_\delta$. And thus, under the same assumptions on the initial data, Moser's iteration technique yields the uniform bound of $u_\delta$. Then, thanks to the local well-posedness result \cite[Theorem 3.1]{finite} we obtain the existence of a  global solution of $(KS)_\delta$. The existence of solutions stated in Theorem \ref{Ex} is then proved using a compactness method; for that purpose we show an additional estimate on $\partial_t u^m_\delta$ which, together with the already derived estimates, guarantees the compactness in space and time of the family $(u_\delta)_{\delta\in (0,1)}$. Finally, in the presence of nonlinear diffusion and under some additional regularity assumption on $\varphi_\delta$, we prove the uniqueness using a classical duality technique.

 \section{Approximated Equations}
 The first equation of (\ref{PKS}) is a quasilinear parabolic equation of degenerate type. Therefore, we cannot expect the system (\ref{PKS}) to have a classical solution at the point where $u$ vanishes. In order to prove Theorem \ref{Ex}, we use a compactness method and introduce the following approximated equations of (KS):
 \begin{equation}
\label{KSd}
(KS)_\delta \left\{
\begin{array}{rclr}
\partial_t u_\delta&=&\mathrm{div}\left(\nabla (u_\delta+\delta)^m-u_\delta \nabla \varphi_\delta\right)& x\in \Omega, t>0,\\
-\Delta \varphi_\delta & =&u_\delta-<u_\delta>&x \in \Omega, t>0,\\
\partial_\nu u_\delta =\partial_\nu \varphi_\delta&=&0 & x\in \partial \Omega, t>0,\\
u_\delta(0, x)&=&u_0(x)& x \in \Omega,
\end{array}
\right.
\end{equation}
where $\delta\in (0,1)$.\\
The main purpose of this section is to construct the time global strong solution of \eqref{KSd}.
\subsection{Existence of global strong solution of $(KS)_\delta$}
\begin{theorem}\label{ex}
For $\delta\in (0,1)$ and $T>0$, we consider an initial condition $u_{0}\in L^\infty(\Omega)$, $u_0\geq0$ and such that $||u_0||_1<M_*$ where $M_*$ is defined in \eqref{mass}. Then $(KS)_\delta$ has a global strong solution $(u_\delta, \varphi_\delta)$ which is bounded in $L^\infty((0,T)\times \Omega)$ for all $T>0$ uniformly with respect to $\delta \in (0,1)$.
\end{theorem}
The starting point of the proof of Theorem \ref{ex} is the following local well-posedness result \cite[Theorem 1.3]{finite}: 
\begin{lemma}
Let the same assumptions as that in Theorem \ref{ex} hold.
There exists a maximal existence time $T^\delta_{max}\in (0,\infty]$ and a unique solution $(u_\delta, \varphi_\delta)$ of $(KS)_\delta$ in $ [0,T^\delta_{max})\times \Omega$. Moreover, $$\mathrm{if}\ T^\delta_{max}<\infty\  \mathrm{then}\ \lim_{t \to T^\delta_{max}}||u_\delta(t, .)||_{\infty}= \infty.$$
In addition $<u_\delta(t)>=<u_0>=M$ for all $t\in [0, T^\delta_{\mathrm{max}})$.
\end{lemma}
To prove Theorem \ref{ex} we need to prove some lemmas which control $L^m$ norm, $L^p$ norm and $L^\infty$ norm of the solution $u_\delta$ of \eqref{KSd}.
\subsection{$L^p$ -estimates, $1\leq p\leq\infty$.}
Our goal is to show that if $||u_0||_{1}$ is small enough then all solutions are global in time and uniformly bounded.\\

Let us first prove the $L^m$ bound for $u_\delta$.
\begin{lemma}\label{le}
Let the same assumptions as that in Theorem \ref{ex} hold and $(u_\delta,\varphi_\delta)$ be the nonnegative maximal solution of $(KS)_\delta$. Then, $u_\delta$ satisfies the following estimate $$||u_\delta(t)||_{m}\leq C_0,\ \mathrm{for\ all} \ t\in [0, T^\delta_{\mathrm{max}}) $$ and $||u_\delta(t)||_1=||u_0||_1$ where $C_0$ is a constant independent of $T_{max}^\delta$ and $\delta$.
\end{lemma}

\begin{proof}
In this proof, the solution to equation (\ref{KSd}) should be denoted by $(u_\delta, \varphi_\delta)$ but for simplicity we drop the index.\\
Let us define the functional $L_\delta$  by
$$L_\delta(u,\varphi)=\int_\Omega (b_\delta(u)+\frac{1}{2}|\nabla\varphi|^2-u\ \varphi)\ \mathrm{d}x,$$
where $$b_\delta(u):= \int_1^u \int_1^z \frac{m(\sigma+\delta)^{m-1}}{\sigma}\ d\sigma\  dz,$$ such that $b_\delta(1)=b_\delta'(1)=0$ and $b(u)\geq 0$. According to \cite{liapunov} it is a Liapunov functional for $(KS)_\delta$. Indeed,
\begin{eqnarray}
\frac{d}{d t}L_\delta(u(t),\varphi(t))&=&\int_\Omega b_\delta'(u)\ \partial_t u\ dx-\int_\Omega \Delta \varphi \ \partial_t \varphi\ dx-\int_\Omega \partial_t u\  \varphi\ dx-\int_\Omega u\ \partial_t \varphi\ dx\nonumber\\
&=&\int_\Omega \partial_t u\ (b_\delta'(u)-\varphi)\ dx-\int_\Omega (\Delta \varphi +u)\ \partial_t \varphi\ dx\nonumber\\
&=&\int_\Omega \mathrm{div} \left(m\ (u+\delta)^{m-1}\ \nabla u-u\ \nabla\varphi\right)\ (b_\delta'(u)-\varphi)\ dx-\int_\Omega <u(t)>\ \partial_t \varphi\ dx\nonumber\\
&=&-\int_\Omega (m\ (u+\delta)^{m-1}\ \nabla u-u\ \nabla  \varphi)\ (b_\delta''(u)\ \nabla u-\nabla \varphi)
\ dx -M\ \frac{d}{dt}\int_\Omega \varphi \ dx\nonumber\\
&=&-\int_\Omega u\ (b_\delta''(u)\ \nabla u-\nabla \varphi)^2\ dx\nonumber\\
&\leq&0.\nonumber
\end{eqnarray}
Then, we can conclude that for all $t \in [0, T^\delta_{max})$ we have $L_\delta(u(t),\varphi(t))\leq L_\delta(u_0,\varphi_0).$
Using Sobolev inequality \eqref{sob}, H\"older inequality, and Young inequality we obtain
\begin{equation*}
\int_\Omega u\ \varphi\ dx \leq ||\varphi||_{2^*}\ ||u||_{\frac{2N}{N+2}}
\leq  C_s^{-1}||\nabla \varphi||_2\ ||u||_{\frac{2N}{N+2}}
\leq  \frac{1}{2}||\nabla \varphi||_2^2+\frac{C_s^{-2}}{2}||u||^2_{\frac{2N}{N+2}}.
\end{equation*}
 Since $\frac{2\ N}{N+2}<m$, and using interpolation inequality we get,
 $$||u||_{\frac{2N}{N+2}}\leq ||u||_1^{\frac{1}{N}}\ ||u||_m^{\frac{N-1}{N}}\leq M^{\frac{1}{N}}\ |\Omega|^{\frac{1}{N}}\ ||u||_m^{\frac{m}{2}}.$$
 Then, $$\int_\Omega u\ \varphi\ \mathrm{d}x\leq \frac{1}{2}||\nabla \varphi||_2^2+\frac{C_s^{-2}}{2}\ M^{\frac{2}{N}}\ |\Omega|^{\frac{2}{N}}\ ||u||_m^m.$$
Substituting this into the Liapunov functional, we find:
\begin{eqnarray}
L_\delta(u,\varphi)&\geq&\int_\Omega (b_\delta(u)+\frac{1}{2}|\nabla\varphi|^2)\ \mathrm{d}x -\frac{1}{2}||\nabla\varphi||^2_2-\frac{C_s^{-2}}{2}\ M^{\frac{2}{N}}\ |\Omega|^{\frac{2}{N}}\ ||u||_m^m\nonumber\\
&\geq&\int_\Omega b_\delta(u)\ \mathrm{d}x-\frac{C_s^{-2}}{2}\ M^{\frac{2}{N}}\ |\Omega|^{\frac{2}{N}}\ ||u||_m^m.\nonumber
\end{eqnarray} 
We next observe that:
\begin{eqnarray}
b_\delta(u)&=&m \int^u_1\int^z_1 \frac{(\delta+s)^{m-1}}{s}\ ds dz
\geq m\int_1^u\int^z_1 s^{m-2}\ ds dz\nonumber\\
&\geq& \frac{u^m}{m-1}-\frac{m}{m-1}u+1
\geq\frac{u^m}{m-1}-\frac{m}{m-1}u.\nonumber
\end{eqnarray}
Then:
\begin{eqnarray}
L_\delta(u,\varphi)&\geq& \frac{1}{m-1}\ ||u||^m_m-\frac{C_s^{-2}}{2}\ |\Omega|^{\frac{2}{N}}\ M^{\frac{2}{N}}\ ||u||_m^m-\frac{m}{m-1}\ M\ |\Omega|\nonumber\\
&=& \left(\frac{1}{m-1}-\frac{C_s^{-2}}{2}\ M^{\frac{2}{N}}\ |\Omega|^{\frac{2}{N}}\right)\ ||u||^m_m-\frac{m}{m-1}\ M\ |\Omega|.\nonumber
\end{eqnarray}
Let us define $\omega_M$ by
$$\omega_M:=\frac{1}{m-1}-\frac{C_s^{-2}}{2}\ M^{\frac{2}{N}}\ |\Omega|^{\frac{2}{N}}= \frac{|\Omega|^{\frac{2}{N}}}{2\ C_s^2}\ (M_*^{\frac{2}{N}}-M^{\frac{2}{N}}).$$
Since $M=||u_0||_1<M_*$, then $\omega_M$ is positive.   
Finally we get, $$L_\delta(u_0,\varphi_0)+\frac{m}{m-1}\ M\ |\Omega|\geq L_\delta(u(t),\varphi(t))+\frac{m}{m-1}\ M\ |\Omega|\geq \omega_M\ ||u(t)||_m^m\ \ \mathrm{for}\ t\in [0,T^\delta_{max}).$$
In addition, we can see that $L_\delta(u_0,\varphi_0)\leq C$ where $C$ is independent of $\delta\in (0,1)$. In fact, 
\begin{eqnarray}
L_\delta(u_0,\varphi_0)&=&\int_\Omega (b_\delta(u_0)+\frac{1}{2}|\nabla\varphi_0|^2-u_0\ \varphi_0)\ \mathrm{d}x,\nonumber
\end{eqnarray}
and, since $(\delta+s)^{m-1}\leq \delta^{m-1}+s^{m-1}\leq 1+s^{m-1}$ we obtain
\begin{eqnarray}
b_\delta(u_0)= m\int_1^{u_0}\int_1^z \frac{(\delta+s)^{m-1}}{s}\ dsdz\leq m\int_1^{u_0}\int_1^z \frac{1+s^{m-1}}{s}\ dsdz\nonumber\\
\leq m(u_0 \ln u_0-u_0+1)+\frac{m}{m-1}\left(\frac{u_0^m}{m}-u_0+1\right).\nonumber
\end{eqnarray}
Using Young inequality we get
\begin{equation}
L_\delta(u_0,\varphi_0)\leq m\ ||u_0||^2_2+m\ |\Omega|+\frac{||u_0||_m^m}{m-1}+\frac{m\ |\Omega|}{m-1} +\frac{1}{2}||\nabla \varphi_0||_2^2+\frac{1}{2}||u_0||^2_2+\frac{1}{2}||\varphi_0||_2^2.\nonumber
\end{equation}
since $u_0\in L^\infty(\Omega)$ and $\varphi_0\in H^1(\Omega)$ we get $L_\delta(u_0,\varphi_0)\leq C$ where C is independent of $\delta$ 
and the proof of the lemma is complete.
\end{proof}
Thanks to Lemma \ref{le}, let us now show that for all $p>m$ the $L^p$ bound for $u_\delta$.
\begin{lemma}\label{li}
Let the same assumptions as that in Theorem \ref{ex} hold. Then for all $T>0$ and all $p\in (1, \infty)$ there exists $C(p, T)$ independent on $\delta$ such that, for all $t\in [0, T^\delta_\mathrm{max})\cap [0,T]$, the solution $(u_\delta, \varphi_\delta)$ to $(KS)_\delta$ satisfies \begin{equation}\label{nop}||u_\delta(t)||_{p}\leq C(p,T),\end{equation}
 and
 \begin{equation}\label{grad}
 \int_0^t\int_\Omega (\delta+u_\delta)^{m-1}\ u_\delta^{p-2}\ |\nabla u_\delta|^2\ dx ds\leq C(p,T).
 \end{equation}
\end{lemma}
To prove the previous lemma we need the following preliminary result \cite{global}.
\begin{lemma}\label{GO}
Consider $0<q_1<q_2\leq 2^*$. There is $C_1$ depending only on $N$ such that
\begin{equation}
\label{gag}
||u||_{q_2}\leq C_{1}^\theta\ ||u||^\theta_{H^1(\Omega)}\ ||u||_{q_1}^{1-\theta},\ \mathrm{for}\ u\in H^1(\Omega),
\end{equation}
with $$\theta= \frac{2N\ (q_2-q_1)}{q_2[(N+2)q_1+2N(1-q_1)]}\in [0,1].$$
\end{lemma}
\begin{proof}
For $u\in H^1(\Omega)$ we have by Sobolev inequality
\begin{equation}\label{sobo}||u||_{2^*}\leq C_N ||u||_{H^1}.\end{equation} By interpolation inequality we have for $0<q_1<q_2\leq 2^*$ 
\begin{equation}\label{inter}||u||_{q_2}\leq ||u||_{2^*}^\theta\ ||u||^{1-\theta}_{q_1},\end{equation} where $\frac{1}{q_2}=\frac{\theta (N-2)}{2N}+\frac{1-\theta}{q_1}$. Hence, substitute \eqref{sobo} into \eqref{inter} and the lemma is proved.
\end{proof}
Now, we recall the following generalized Poincar\'e inequality.
\begin{lemma}\label{poincare}
For $u\in H^1(\Omega)$ we have for $0< q_1\leq1$ the following inequality
$$||u||^2_{H^1}\leq C_2(q_1)\ (||\nabla u||_2^2+||u||_{q_1}^2),$$
where $C_2$ depends only on $\Omega$ and $q_1$.
\end{lemma}
Now using  the last two lemmas, let us prove Lemma \ref{li}.
\begin{proof}
In this proof, the solution to equation (\ref{KSd}) should be denoted by $(u_\delta, \varphi_\delta)$ but for simplicity we drop the index.\\
We choose $p>1$, $K\geq0$ and we multiply the first equation in \eqref{KSd} by $(u-K)_+^{p-1}$ and integrate by parts using the boundary conditions for $u$ and $\varphi$ to see that
\begin{eqnarray}
\frac{1}{p}\frac{\mathrm{d}}{\mathrm{dt}}||(u-K)_+||_p^p
&=&-m (p-1)\int_\Omega (\delta+u)^{m-1}\ (u-K)_+^{p-2}\ |\nabla u|^2\ \mathrm{d}x\nonumber\\
&+&(p-1)\int_\Omega u\ \nabla\varphi\ (u-K)_+^{p-2}\cdot \nabla u\ \mathrm{d}x\nonumber\\
&=&-m(p-1)\int_\Omega (\delta+u-K+K)^{m-1}\ (u-K)_+^{p-2}\ |\nabla u|^2\ \mathrm{d}x\nonumber\\
&+&(p-1) \int_\Omega (u-K+K)\ \nabla \varphi\cdot (u-K)_+^{p-2}\ \nabla u\ \mathrm{d}x\nonumber\\
&\leq & -m(p-1)\int_\Omega (u+\delta-K)^{m-1}\ (u-K)_+^{p-2}\ |\nabla u|^2\ \mathrm{d}x\nonumber\\
&+&(p-1)\int_\Omega (u-K)_+^{p-1}\ \nabla \varphi\cdot \nabla u\ \mathrm{d}x
+ (p-1) K\int_\Omega \nabla\varphi\ (u-K)_+^{p-2}\cdot \nabla u\ \mathrm{d}x\nonumber\\
&\leq &- m(p-1)\int_\Omega (u+\delta-K)^{m-1}\ (u-K)_+^{p-2}\ |\nabla u|^2\ \mathrm{d}x\nonumber\\
&-&\frac{p-1}{p}\int_\Omega (u-K)_+^p\ \Delta\varphi\  \mathrm{d}x
-K\int_\Omega(u-K)_+^{p-1}\ \Delta \varphi\ \mathrm{d}x\nonumber\\
&\leq&-m(p-1)\int_\Omega (u+\delta-K)^{m-1}\ (u-K)_+^{p-2}\ |\nabla u|^2\ \mathrm{d}x+(I),\nonumber
\end{eqnarray}
where, thanks to the second equation in \eqref{KSd}, 
\begin{eqnarray}
(I)&=&\frac{p-1}{p}\int_\Omega (u-K)_+^p\ (u-M)\ \mathrm{d}x+K\int_\Omega (u-K)_+^{p-1}\ (u-M)\ \mathrm{d}x\nonumber\\
&=& \frac{p-1}{p}||(u-K)_+||_{p+1}^{p+1}+\frac{p-1}{p}(K-M)||(u-K)_+||_p^p\nonumber\\
&+& K||(u-K)_+||_p^p+K(K-M) ||(u-K)_+||^{p-1}_{p-1}\nonumber\\
&\leq& K^2\ ||(u-K)_+||_{p-1}^{p-1}+ 2 K\ ||(u-K)_+||^p_p+||(u-K)_+||^{p+1}_{p+1}.\nonumber
\end{eqnarray}
Since for $a>0$ and $b>0$ we have $a^{p-1}b\leq a^{p+1}+b^{\frac{p+1}{2}}$ and $a^p b\leq a^{p+1}+b^{p+1}$ then,
\begin{eqnarray}
(I)&\leq& 3||(u-K)_+||^{p+1}_{p+1}+(2K)^{p+1}+K^{p+1},
\end{eqnarray}
and we get
\begin{eqnarray}
\frac{\mathrm{d}}{\mathrm{dt}} ||(u-K)_+||_p^p&\leq& -m(p-1)\int_\Omega (u+\delta-K)^{m-1}\ (u-K)_+^{p-2}\ |\nabla u|^2\ \mathrm{d}x\nonumber\\
&+&3p||(u-K)_+||^{p+1}_{p+1}+C_p\ K^{p+1},\label{estimatio} \end{eqnarray}
for all $t \in [0, T^\delta_{\mathrm{max}})$.\\
 The term $||(u-K)_+||_{p+1}^{p+1}$ can be estimated with the help of  Lemma \ref{GO} and Lemma \ref{poincare}. Assuming now that $p>2$ we remark that $0<\frac{2}{p+m-1}\leq1$ and $1<\frac{2(p+1)}{p+m-1}= \frac{2N}{N-2}\ \frac{1+p}{1+\frac{Np}{N-2}}\leq \frac{2N}{N-2}$, then thanks to Lemma \ref{GO} and Lemma \ref{poincare} we obtain 
\begin{eqnarray}\label{GNM}
||(u-K)_+^{\frac{p+m-1}{2}}||^{\frac{2(p+1)}{p+m-1}}_{{\frac{2(p+1)}{p+m-1}}}&\leq & C(p)\ \left(||\nabla (u-K)_+^{\frac{p+m-1}{2}}||^{\frac{2(p+1)}{p+m-1}\theta}_{2}\  ||(u-K)_+^{\frac{p+m-1}{2}}||^{\frac{2(p+1)}{p+m-1}(1-\theta)}_{{\frac{2}{p+m-1}}}\right.\nonumber\\
&+&\left.||(u-K)_+^{\frac{p+m-1}{2}}||^{\frac{2(p+1)}{p+m-1}}_{{\frac{2}{p+m-1}}}\right),
\end{eqnarray}
where
\begin{equation}
\theta=\frac{p+m-1}{p+1}\ \in(0,1).
\end{equation}
Since
\begin{equation}
||(u-K)_+^{\frac{p+m-1}{2}}||^{\frac{2(p+1)}{p+m-1}}_{{\frac{2(p+1)}{p+m-1}}}=\int_\Omega (u-K)_+^{p+1}\ \mathrm{d}x
= ||(u-K)_+||^{p+1}_{p+1},\label{1}
\end{equation}
\begin{equation}
||(u-K)_+^{\frac{p+m-1}{2}}||^{\frac{2(p+1)}{p+m-1}(1-\theta)}_{{\frac{2}{p+m-1}}}=\left(\int _\Omega(u-K)_+\ \mathrm{d}x\right)^{(p+1)(1-\theta)}
=||(u-K)_+||_{1}^{\frac{2}{N}},\label{2}
\end{equation}
and by Lemma \ref{le}
\begin{equation}
||(u-K)_+||_1=\int _{u\geq K}(u-K)\ \mathrm{d}x
\leq  \frac{1}{K^{m-1}}\int_{u\geq K} K^{m-1}\ u\ dx
\leq \frac{||u||_m^m}{K^{m-1}}\leq \frac{C_0^m}{K^{m-1}},\label{3}
\end{equation}
we substitute \eqref{1}, \eqref{2} and \eqref{3} into \eqref{GNM} and obtain
\begin{equation}
||(u-K)_+||^{p+1}_{p+1}\leq C_3(p)\left\{   ||\nabla(u-K)_+^{\frac{m+p-1}{2}}||^2_2\  K^{\frac{-2(m-1)}{N}}+ K^{-(m-1)(p+1)}\right\}.\end{equation}
We may choose $K=K_*$ large  enough such that
$$3\ p\ C_3(p)\ K_*^{\frac{-2(m-1)}{N}}\leq \frac{2\ p\ (p-1)\ m }{(m+p-1)^2},$$
Hence
$$\frac{\mathrm{d}}{\mathrm{dt}}||(u-K_*)_+||_p^p\leq C(p)\ K_*^{p+1},$$
so that 
$$||(u(t)-K_*)_+||^p_p\leq C(p)\ t+||u_0||^p_p, \ \mathrm{for}\ t\in[0, T^\delta_{\mathrm{max}}).$$
As
\begin{eqnarray}
\int_\Omega |u|^p\ \mathrm{d}x&\leq&\int _{u< 2 K_*} (2\ K_*)^{p-1}\ |u|\ \mathrm{d}x+\int _{u\geq 2 K_*}|u-K_*+K_*|^p\ \mathrm{d}x\nonumber\\
&\leq& (2K_*)^{p-1}\ M+\int_ {u\geq 2 K_*}(2\ |u-K_*|)^p\ \mathrm{d}x,\nonumber\\
&\leq& (2K_*)^{p-1}\ M+2^p\ ||(u-K_*)_+||^p_p,\nonumber
\end{eqnarray}
the previous inequality warrants that
\begin{equation}||u(t)||_{p}\leq C(p,T),\ \ t\in [0, T_{\mathrm{max}})\cap [0,T],\end{equation}
where $C(p,T)$ is a constant independent of $\delta$.\\

We next take $K=0$ in \eqref{estimatio}, integrate with respect to time and use \eqref{nop} to obtain \eqref{grad}.\end{proof}
Thanks to Lemma \ref{li}, we can improve the regularity of $\varphi_\delta$.
\begin{lemma}\label{dx}
Let the same assumptions as that in Theorem \ref{ex} hold, the solution $\varphi_\delta$ satisfies
$$||\nabla\varphi_\delta(t)||_{\infty}\leq L(T),\ t\in [0,T_{\mathrm{max}}^\delta)\cap [0,T]$$ where $T>0$ and $L$ is a positive constant independent of $\delta$.
\end{lemma}
\begin{proof}
Using standard elliptic regularity estimates for $\varphi_\delta$, we infer from Lemma \ref{li} that given $T>0$, and $p\in (1, \infty)$, there is $C(p, T)$ such that 
$$||\varphi_\delta(t)||_{W^{2,p}}\leq C(p)\ ||u_\delta(t)||_p\leq C(p, T),\ \mathrm{for}\ t\in [0, T_{\mathrm{max}})\cap [0,T].$$
Lemma \ref{dx} then readily follows from Sobolev embedding theorem upon choosing $p>N$.
\end{proof}
\begin{lemma}\label{app}
Let $N\geq 3$, $r\geq 4$, $u\in L^{\frac{r}{4}}(\Omega)$, and $u^{\frac{r+m-1}{2}}\in H^1(\Omega)$. Then it holds that
\begin{equation}\label{ur}
||u||_{r}\leq C_1^{\frac{2 \theta}{r+m-1}}\ ||u||^{1-\theta}_{\frac{r}{4}}\ ||u^{\frac{r+m-1}{2}}||_{H^1}^{\frac{2\theta}{r+m-1}}
\end{equation}
with
\begin{equation}\label{ou}\theta= \frac{3\ N (r+m-1)}{(3N+2)r+4N(m-1)}\ \in (0,1).\end{equation} 
\end{lemma}
\begin{proof} For $r\geq4$, we can see that
\begin{equation}
||u||_r= \left( \int_\Omega (u^{\frac{r+m-1}{2}})^{\frac{2r}{r+m-1}}\ dx\right)^{\frac{1}{r}}
= ||u^{\frac{r+m-1}{2}}||^{\frac{2}{r+m-1}}_{\frac{2r}{r+m-1}},\nonumber
\end{equation}
and $$\frac{r}{2(r+m-1)}<1<\frac{2r}{r+m-1}<2<\frac{2N}{N-2}.$$ By Lemma \ref{GO},
\begin{equation*}
||u||_r= ||u^{\frac{r+m-1}{2}}||^{\frac{2}{r+m-1}}_{\frac{2r}{r+m-1}}
\leq \left( C_1^\theta \ ||u^{\frac{r+m-1}{2}}||^\theta_{H^1(\Omega)}\ ||u^{\frac{r+m-1}{2}}||_{\frac{r}{2(r+m-1)}}^{1-\theta}\right)^{\frac{2}{r+m-1}}
\end{equation*}
and 
\begin{eqnarray}
\theta &=& \frac{2N\ (\frac{2r}{r+m-1}-\frac{r}{2\ (r+m-1)})} {\frac{2r}{r+m-1}\ \left( 2N(1-\frac{r}{2(r+m-1)})+(N+2)\ \frac{r}{2\ (r+m-1)}\right)}\nonumber\\
&=& \frac{3N\ (r+m-1)}{(3N+2)\ r+4N\ (m-1)}
\in (0,1).\nonumber
\end{eqnarray}
In addition, we have
\begin{equation*}
||u^{\frac{r+m-1}{2}}||_{\frac{r}{2(r+m-1)}}= \left( \int_\Omega |u|^{\frac{r+m-1}{2}\ \frac{r}{2(r+m-1)}}\ dx\right)^{\frac{2(r+m-1)}{r}}
= ||u||_{\frac{r}{4}}^{\frac{r+m-1}{2}},
\end{equation*}
and we obtain \eqref{ur}.
\end{proof}
We are now in a position to prove the uniform $L^\infty(\Omega)$  bound for $u_\delta$.
\begin{lemma}\label{inf}
Let the same assumptions as that in Theorem \ref{ex} hold, and $(u_\delta, \varphi_\delta)$ be the nonnegative maximal solution of \eqref{KSd}. For all $T>0$, there is $C_\infty(T)$ such that
$$||u_\delta(t)||_{\infty}\leq C_\infty(T), \ \mathrm{for\ all}\ t\in [0,T_{\mathrm{max}}^\delta)\cap [0,T],$$
where $C_\infty(T)$ is a positive constant independent on $\delta$. 
\end{lemma}
\begin{proof}
In this proof we omit the index $\delta$, and we employ Moser's iteration technique developed in \cite{lpb, time} to show the uniform norm bound for $u$.\\
We multiply the first equation in \eqref{KSd} by $u^{r-1}$, where $r\geq4$, and integrate it over $\Omega$. Then, we have
\begin{eqnarray}
\frac{\mathrm{d}}{\mathrm{d}t}\frac{||u||_r^r}{r}
&=&-\int_\Omega (\nabla (u+\delta)^m-u\ \nabla\varphi)\cdot \nabla u^{r-1}\ \mathrm{d}x\nonumber\\
&=&-m(r-1)\int_\Omega (u+\delta)^{m-1}\ u^{r-2}\ |\nabla u|^2\ \mathrm{d}x+(r-1)\int_\Omega u^{r-1}\ \nabla \varphi\cdot \nabla u\ \mathrm{d}x\nonumber\\
&\leq&-m(r-1)\int_\Omega u^{m+r-3}\ |\nabla u|^2\ \mathrm{d}x+(r-1)\int_\Omega u^{r-1}\ \nabla\varphi\cdot \nabla u\ \mathrm{d}x.\nonumber
\end{eqnarray}
By Young's inequality and Lemma \ref{dx},
\begin{eqnarray}
\frac{1}{r}\frac{\mathrm{d}}{\mathrm{d}t} ||u||^r_{r}&\leq& \frac{-4 m (r-1)}{(r+m-1)^2}\int_\Omega |\nabla u^{\frac{r+m-1}{2}}|^2\ \mathrm{d}x+\frac{2(r-1)\ ||\nabla \varphi||_\infty}{(r+m-1)}\int_\Omega u^{\frac{r-m+1}{2}}\ |\nabla u^{\frac{r+m-1}{2}}|\ \mathrm{d}x\nonumber\\
&\leq& \frac{-2 m (r-1)}{(r+m-1)^2}||\nabla u^{\frac{r+m-1}{2}}||^2_{2}+\frac{r-1}{2m}||\nabla \varphi||^2_\infty \int_\Omega u^{r-m+1}\ \mathrm{d}x\nonumber\\
&\leq &\frac{-2 m (r-1)}{(r+m-1)^2}||\nabla u^{\frac{r+m-1}{2}}||^2_{2}+C(T)\ r\ \int_\Omega u^{r-m+1}\ \mathrm{d}x.\nonumber
\end{eqnarray}
Using H\"older and Young inequalities and Lemma \ref{le} we obtain
\begin{eqnarray}
\frac{1}{r}\frac{\mathrm{d}}{\mathrm{d}t} ||u||^r_r&\leq& \frac{-2 m\ (r-1)}{(r+m-1)^2}||\nabla u^{\frac{r+m-1}{2}}||^2_{2}+r\ C(T)\ ||u||_{1}^{\frac{m-1}{r-1}}\ ||u||_{r}^{\frac{r(r-m)}{r-1}}\nonumber\\
&\leq&\frac{-2 m\  (r-1)}{(r+m-1)^2}||\nabla u^{\frac{r+m-1}{2}}||^2_{2}+C^r+r^2 ||u||^r_{r},\label{dd}
\end{eqnarray}
where we have used that $r^{\frac{r-1}{r-m}}\leq r^2$ for $r\geq 4$.\\
By Lemma \ref{app}, we have for $r\geq4$
\begin{equation}\label{Ga}
||u||^r_{r}\leq C_1^{\frac{2\ r\ \theta}{r+m-1}}\ ||u||_{{\frac{r}{4}}}^{r(1-\theta)}\ ||u^{\frac{r+m-1}{2}}||^{\frac{2\ r\ \theta}{r+m-1}}_{H^1},
\end{equation}
where 
$$\theta=\frac{3\ N (r+m-1)}{(3N+2)r+4N(m-1)}<1.$$
Therefore, Young inequality and \eqref{Ga} yield that
\begin{eqnarray}
2\ r^2\ ||u||^r_{r}&\leq& 2\ r^2\  C_1^{\frac{2\ r\ \theta}{r+m-1}}\ ||u||_{{\frac{r}{4}}}^{r(1-\theta)}||\  u^{\frac{r+m-1}{2}}||^{\frac{2\ r\ \theta}{r+m-1}}_{H^1}\nonumber\\
&\leq& \frac{\theta\ r}{r+m-1}\ \frac{m\ (r-1)}{(r+m-1)^2}\ \frac{r+m-1}{\theta\ r\ C_2(1)}\ ||u^{\frac{r+m-1}{2}}||^2_{H^1}\nonumber\\
&+&\frac{r+m-1-\theta r}{r+m-1}\ \left( C_2(1)\ \frac{\theta (r+m-1)r}{m(r-1)}\right)^{\frac{\theta r}{r(1-\theta)+m-1}}\nonumber\\
&\times& (2\ r^2)^{\frac{(r+m-1)}{r(1-\theta)+m-1}}\ C_1^{\frac{2\ \theta r}{r(1-\theta)+m-1}}\ ||u||_{\frac{r}{4}}^{(1-\theta) r\frac{(r+m-1)}{r(1-\theta)+m-1}},\nonumber
\end{eqnarray}
where $C_2(1)$ is the Poincar\'e constant defined in Lemma \ref{poincare}. Then we obtain
\begin{eqnarray}
2\ r^2\ ||u||^r_r&\leq& \frac{m\ (r-1)}{C_2(1)\ (r+m-1)^2}||u^{\frac{r+m-1}{2}}||^2_{H^1}\nonumber\\
&+& C_1^{\frac{\theta\ r}{r(1-\theta)+m-1}}\ \ 2^{\frac{(r+m-1)}{r(1-\theta)+m-1}}\ r^{\frac{2(r+m-1)+\theta r}{r(1-\theta)+m-1}}\ ||u||_{\frac{r}{4}}^{\frac{(1-\theta)r(r+m-1)}{r(1-\theta)+m-1}}.\nonumber
\end{eqnarray}
Now, since $N>2$, which gives $4N\geq 3N+2$, we find the following upper bound for $\theta$
\begin{equation}\label{theta}\theta\leq\frac{3 N}{3N+2}\end{equation}
In addition,
\begin{equation}
\frac{\theta\ r}{r(1-\theta)+m-1}\leq \frac{\theta}{1-\theta}=-1+\frac{1}{1-\theta}\leq\frac{3N}{2},
\end{equation}
\begin{equation}
\frac{r+m-1}{r(1-\theta)+m-1}\leq \frac{r+m-1}{(1-\theta) (r+m-1)}\leq \frac{1}{1-\theta}\leq \frac{3N+2}{2},\label{nn}
\end{equation}
and
\begin{equation}
\frac{2(r+m-1)+\theta r}{r(1-\theta)+m-1}\leq \frac{2+\theta}{1-\theta}\leq 9N+4.
\end{equation}
As $C_1\geq 1$ and $r\geq1$, we get 
\begin{eqnarray}
2\ r^2\ ||u||^r_r&\leq& \frac{m\ (r-1)}{C_2(1)\ (r+m-1)^2}||u^{\frac{r+m-1}{2}}||^2_{H^1}
+ C\ r^{9N+4}\ ||u||_{\frac{r}{4}}^{\frac{(1-\theta)r(r+m-1)}{r(1-\theta)+m-1}}.\label{cc}
\end{eqnarray}
Using Lemma \ref{poincare} we have
\begin{equation}\label{bb}
||u^{\frac{r+m-1}{2}}||^2_{H^1}\leq C_2(1)\ \left( ||\nabla u^{\frac{r+m-1}{2}}||_2^2+||u^{\frac{r+m-1}{2}}||^2_1\right).
\end{equation}
Using H\"older inequality, Young inequality and Lemma \ref{le}, we get 
\begin{equation*}
||u^{\frac{r+m-1}{2}}||^2_1= ||u||^{r+m-1}_{\frac{r+m-1}{2}}
\leq ||u||_r^{r\ \frac{r+m-3}{r-1}}\ ||u||_1^{\frac{r-m+1}{r-1}}\leq ||u||_r^{r\ \frac{r+m-3}{r-1}}\ ||u_0||_1^{\frac{r-m+1}{r-1}},
\end{equation*}
then,
\begin{eqnarray}
\frac{m\ (r-1)}{(r+m-1)^2}||u^{\frac{r+m-1}{2}}||^2_1&\leq&(r-1)^{\frac{r-1}{r+m-3}}\ \frac{r+m-3}{r-1}\ ||u||_r^r\nonumber\\
&+&\frac{2-m}{r-1}\ \left( \frac{m}{(r+m-1)^2} ||u_0||_1^{\frac{r+m-1}{r-1}}\right)^{\frac{r-1}{2-m}}\nonumber\\
&\leq& r^2 ||u||_r^r+\left(\frac{m}{(r+m-1)^2}\ ||u_0||_1^{\frac{r+m-1}{r-1}}\right)^{\frac{r-1}{2-m}}\nonumber\\
&\leq& r^2||u||^r_r+C_4^r.\label{aa}
\end{eqnarray}
Now substituting \eqref{aa} and \eqref{bb} into \eqref{cc} we get
\begin{eqnarray}
2\ r^2\ ||u||^r_r &\leq&\frac{m\ (r-1)}{(r+m-1)^2}\ \left( ||\nabla u^{\frac{r+m-1}{2}}||_2^2+||u^{\frac{r+m-1}{2}}||^2_1\right)+C\ r^{9N+4}\ ||u||_{\frac{r}{4}}^{\frac{(1-\theta)r(r+m-1)}{r(1-\theta)+m-1}}\nonumber\\
&\leq&\frac{m\ (r-1)}{(r+m-1)^2}\ ||\nabla u^{\frac{r+m-1}{2}}||_2^2+r^2||u||^r_r+ C_4^r+C\ r^{9N+4}\ ||u||_{\frac{r}{4}}^{\frac{(1-\theta)r(r+m-1)}{r(1-\theta)+m-1}},\nonumber 
\end{eqnarray}
hence
$$r^2\ ||u||^r_r\leq \frac{m\ (r-1)}{(r+m-1)^2} \ ||\nabla u^{\frac{r+m-1}{2}}||_2^2+ C_4^r+C\ r^{9N+4}\ ||u||_{\frac{r}{4}}^{\frac{r\ (1-\theta)(r+m-1)}{r(1-\theta)+m-1}}.$$
We apply Young inequality again to the last term of the above inequality. It is easy to see that \begin{equation*}
\frac{2}{3N+2}\leq 1-\theta\leq\frac{(1-\theta)\ (r+m-1)}{r(1-\theta)+m-1}=\frac{(1-\theta)r+(1-\theta)(m-1)}{r(1-\theta)+m-1}<1,
\end{equation*}
so that
\begin{equation}\label{fin}
r^2 ||u||_{r}^r\leq \frac{m\  (r-1)}{(r+m-1)^2}||\nabla u^{\frac{r+m-1}{2}}||^2_{2}+ C_4^r+
1+\left( C\ r^{9N+4}\right )^{3N+1} ||u||^r_{{\frac{r}{4}}},
\end{equation}
for any $r\in [4 , \infty)$.\\

Substituting \eqref{fin} into \eqref{dd} we end up with
\begin{equation}\label{end}
\frac{\mathrm{d}}{\mathrm{dt}} ||u||^r_{r}\leq r\ C_4^r+r+r\ \left( C\ r^{9N+4}\right )^{3N+1} ||u||^r_{{\frac{r}{4}}}\leq C_5^r+C r^\alpha \ ||u||^r_{\frac{r}{4}},
\end{equation}
for any $r\in [4 , \infty)$, where $\alpha= (9N+4)(3N+1)+1$. After integrating \eqref{end} from $0$ to $t$, we obtain the $L^r$ estimate for $u$ as follows:
\begin{equation}\label{lol}
\sup_{0<t<T} ||u(t)||^r_{r}\leq ||u_0||^r_{r}+ T\ C_5^r
+C\ r^\alpha\ T \sup_{0<t<T}||u(t)||^r_{{\frac{r}{4}}}.
\end{equation}
Since $$||u_0||_r\leq ||u_0||^{\frac{r-1}{r}}_\infty\ ||u_0||^{\frac{1}{r}}_1\leq C_6,$$
then
\begin{equation}
\sup_{0<t<T} ||u(t)||^r_{r}\leq C_7(T)\ r^\alpha \max\left\{C_6, \sup_{0<t<T}\ ||u(t)||_{\frac{r}{4}}\right\}^r,
\end{equation}
and we obtain for $r\geq4$
\begin{equation}\label{sup}\sup_{0<t<T} ||u(t)||_{r}\leq C_7(T)^{\frac{1}{r}}\ r^{\frac{\alpha}{r}} \max\left\{C_6, \sup_{0<t<T}\ ||u(t)||_{\frac{r}{4}}\right\}.\end{equation}
We are now in a position to derive the claimed $L^\infty$ estimate. To this end, we set
$$\alpha_p:=\max \left\{ C_6, \sup_{0<t<T}||u(t)||_{{4^p}}\right\}$$ for $p\geq 0$. Then we take $r=4^p$ with $p\geq 0$ in \eqref{sup} which reads
\begin{eqnarray}
\alpha_p&\leq& 4^{\frac{\alpha\ p}{4^p}}\ C_7(T)^{\frac{1}{4^p}}\ \max\left\{C_6, \sup_{0<t<T}||u(t)||_{{4^{p-1}}}\right\},\nonumber\\
&\leq& 4^{\frac{\alpha}{2^p}}\ C_7(T)^{\frac{1}{4^p}}\ \alpha_{p-1}\nonumber
\end{eqnarray}
since $p\leq 2^p$ for $p\geq 1$. Arguing by induction we conclude that 
$$\alpha_p\leq 4^{\alpha \sum_{k=1}^{p} 2^{-k}}\ C_7(T)^{\sum_{k=1}^{p} 4^{-k}}\ \alpha_0.$$
 Then by using Lemma \ref{le} we get
\begin{equation*}
\sup_{0<t<T} ||u(t)||_{{4^p}}\leq 4^{\alpha}\ C_7(T)\ \alpha_0\leq C_8(T).
\end{equation*}
Consequently, by letting $p$ tend to $\infty$, we see that $u\in L^\infty((0,T)\times \Omega)$ and
\begin{equation}
\sup_{0<t<T} ||u(t)||_{\infty}\leq C_8(T).
\end{equation}

Since the right hand side is independent of $\delta$, we have proved the lemma.
\end{proof}
\begin{lemma}\label{est} Let the same assumptions as that in Theorem \ref{ex} hold, and $(u_\delta, \varphi_\delta)$ be the solution to \eqref{KSd}. Then for all $T>0$ there is $C_9(T)$ such that the solution $u_\delta$ satisfies the following derivation estimate
$$\int_0^T ||\partial_t u^m_\delta||_{(W^{1, N+1})'}\ dt\leq C_9(T).$$
\end{lemma}
\begin{proof}
Consider $\psi \in W^{1,N+1}(\Omega)$ and $t\in (0,T)$, we have
\begin{eqnarray*}
&&\left\lvert\int_\Omega m\ u_\delta^{m-1}(t)\ \partial_t u_\delta(t)\ \psi\ dx\right\lvert\\
&=&m\left\lvert\int_\Omega \nabla(u_\delta^{m-1}\ \psi)\cdot (\nabla u_\delta^m-u_\delta\ \nabla \varphi_\delta)\ dx\right\rvert\nonumber\\
&=& m\left\lvert\int_\Omega (u_\delta^{m-1}\ \nabla \psi+\psi\ \nabla u_\delta^{m-1})\cdot (\nabla u_\delta^m-u_\delta\ \nabla \varphi_\delta)\ dx\right\rvert\nonumber\\
&\leq & m\ \int_\Omega \left[ u_\delta^{m-1}\ |\nabla u_\delta^m|\ |\nabla \psi|+ u_\delta^m \ |\nabla \psi|\ |\nabla \varphi_\delta|\right. \nonumber\\
&&+\left.|\psi|\ m(m-1)\ u_\delta^{2m-3}\ |\nabla u_\delta|^2+|\psi|(m-1)u_\delta^{m-1}\ |\nabla u_\delta|\ |\nabla \varphi_\delta|\right]dx\nonumber\\
&\leq & m\ \left[ ||u_\delta||_\infty^{m-1}\ ||\nabla u_\delta^m||_2\ ||\nabla \psi||_2+||\nabla \psi||_2\ ||u_\delta||_\infty^m ||\nabla \varphi_\delta||_\infty |\Omega|^{\frac{1}{2}}\right.\nonumber\\
&&+\left. ||\psi||_\infty\ \frac{4 m(m-1)}{(2m-1)^2} ||\nabla u_\delta^{m- \frac{1}{2}}||_2^2+||\psi||_2 \frac{m-1}{m}||\nabla u_\delta^m||_2\ ||\nabla \varphi_\delta||_\infty \right].\nonumber
\end{eqnarray*}
Using Lemma \ref{app}, Lemma \ref{inf}, and the embedding of $W^{1, N+1}(\Omega)$ in $L^\infty(\Omega)$, we end up with
$$\left\lvert<\partial_tu_\delta^m(t), \psi>\right\lvert\leq C(T) \left( ||\nabla u_\delta(t)^m||_2+||\nabla u_\delta^{m-\frac{1}{2}}(t)||_2^2+1\right)\ ||\psi||_{W^{1,N+1}},$$
and a duality argument gives
$$||\partial_t u_\delta^m(t)||_{(W^{1, N+1})'}\leq C(T) \left( ||\nabla u^m_\delta(t)||_2+||\nabla u_\delta^{m-\frac{1}{2}}(t)||_2^2+1\right).$$ Integrating the above inequality over $(0,T)$ and using Lemma \ref{li} with $p=2$ and $p=m$ give Lemma \ref{est}.
\end{proof}

\section{Proof of Theorem \ref{Ex}}
\subsection{Existence}
In this section, we assume that $u_0$ is a nonnegative function in $L^\infty(\Omega)$ satisfying \eqref{cond}. For $\delta \in (0,1)$, $(u_\delta, \varphi_\delta)$ denotes the solution to $(KS)_\delta$ constructed in Section 3. To prove existence of a weak solution, we use a compactness method. For that purpose, we first study the compactness properties of $(u_\delta, \varphi_\delta)_\delta$.
\begin{lemma} \label{conv}
There are functions $u$ and $\varphi$ and a sequence $(\delta_n)_{n\geq 1}$, $\delta_n\rightarrow 0$, such that, for all $T>0$ and $p\in (1, \infty),$
\begin{equation}
\label{za}
u_{\delta_n}\longrightarrow u, \ \mathrm{in}\ L^p((0,T)\times \Omega)\ \mathrm{as}\ \delta_n\rightarrow 0,\end{equation}
\begin{equation}
\label{zaa}
\varphi_{\delta_n}\longrightarrow \varphi, \ \mathrm{in}\ L^p((0,T); W^{2,p}(\Omega))\ \mathrm{as}\ \delta_n\rightarrow 0.
\end{equation}
In addition, $u\in L^\infty((0,T)\times \Omega)$ for all $T>0$ and is nonnegative.
\end{lemma}
\begin{proof}
Thanks to Lemma \ref{li} and Lemma \ref{inf}, $(u_\delta^m)_\delta$ is bounded in $L^2((0,T); H^1(\Omega))$ while $(\partial_t u_\delta^m)_\delta$ is bounded in $L^1((0,T); (W^{1, N+1})'(\Omega))$ by Lemma \ref {est}. \\
Since $H^1(\Omega)$ is compactly embedded in $L^2(\Omega)$ and $L^2(\Omega)$ is continuously embedded in $(W^{1, N+1})'(\Omega)$, it follows from \cite[corollary 4]{compact} that $(u_\delta^m)$ is compact in $L^2((0,T)\times \Omega)$ for all $T>0$. Since $r\longmapsto r^{\frac{1}{m}}$ is $\frac{1}{m}$-H\"older continuous, it is easy to check that the previous compactness property implies that $(u_\delta)$ is compact in $L^{2m}((0,T)\times \Omega)$ for all $T>0$. There are thus a function $u\in L^{2m}((0,T)\times \Omega)$ for all $T>0$ and a sequence $(\delta_n)_{n\geq1}$ such that
\begin{equation}
\label {Ze}
u_{\delta_n}\longrightarrow u\ \mathrm{in} \ L^{2m}((0,T)\times \Omega)\ \mathrm{as}\ \delta_n\rightarrow 0, 
\end{equation}
for all $T>0$, owing to Lemma \ref{inf}, we may also assume that
\begin{equation}
\label{Zee}
u_{\delta_n}\stackrel{\ast}\rightharpoonup u\ \mathrm{in}\ L^\infty((0,T)\times \Omega)\ \mathrm{as}\ \delta_n\rightarrow 0.
\end{equation}
for all $T>0$. It readily follows from \eqref{Ze} and \eqref{Zee}, and H\"older inequality that \eqref{za} holds true. Since elliptic regularity ensure that
$$||\varphi_{\delta_k}-\varphi_{\delta_n}||_{W^{2,p}}\leq C(p)\ ||u_{\delta_k}-u_{\delta_n}||_p,$$
for all $k\geq 1$, $n\geq 1$, and $p\in (1, \infty)$, a straightforward consequence of \eqref{za} is that $(\varphi_{\delta_n})_{n\geq 1}$ is a Cauchy sequence in $L^p((0,T); W^{2, p}(\Omega))$ and thus converges to some function $\varphi$ in that space. Finally, the nonnegativity of $u$ follows easily from that of $u_{\delta_n}$ by \eqref{za}.
\end{proof}
\begin {proof}[Proof of Theorem \ref{Ex} (existence)]
It remains to identify the equations solved by the limit $(u, \varphi)$ of $(u_{\delta_n}, \varphi_{\delta_n})_{n\geq1}$ constructed in Lemma \ref{conv}. To this end we first note that , owing to \eqref{za} and the boundedness of $(u_{\delta_n})_n$ and $u$ in $L^\infty((0,T)\times \Omega)$, we have 
\begin{equation}
\label{zu}
u^m_{\delta_n}\longrightarrow u^m\ \mathrm{in}\ L^p((0,T)\times \Omega)\ \mathrm{as}\ \delta_n\rightarrow 0,
\end{equation}
for all $T>0$. Since $(\nabla(u_{\delta_n}+\delta_n)^{\frac{m+1}{2}})_{n\geq1}$ and $(\nabla u_{\delta_n}^m)_{n\geq1}$ are bounded in $L^2((0,T)\times \Omega)$ for all $T>0$ by Lemma \ref{li} with $p=2$ and $p=m+1$, we may extract a further subsequence (not relabeled) such that
\begin{equation}
\label{zr}
\nabla(u_{\delta_n}+\delta_n)^{\frac{m+1}{2}}\rightharpoonup \nabla u^{\frac{m+1}{2}}\ \mathrm{in}\ L^2((0,T)\times \Omega),
\end{equation}
\begin{equation}
\label{zrr}
\nabla u_{\delta_n}^m\rightharpoonup \nabla u^m \ \mathrm{in}\ L^2((0,T)\times \Omega),
\end{equation}
for all $T>0$. Then if $\psi\in L^4((0,T)\times \Omega; \mathbb{R}^N)$,\\

$
\begin{array}{l}
\displaystyle
\left\lvert\int_0^T\int_\Omega \psi\cdot \left[\nabla(u_{\delta_n}+\delta_n)^m-\nabla u^m\right]\ dx ds\right\rvert\nonumber\\
\displaystyle=\frac{2}{m+1}\left\lvert \int_0^T\int_\Omega \psi\cdot\left[ (u_{\delta_n}+\delta_n)^{\frac{m-1}{2}}\nabla(u_{\delta_n}+\delta_n)^{\frac{m+1}{2}}-u^{\frac{m-1}{2}}\nabla u^{\frac{m+1}{2}}\right]\ dxds\right\rvert\nonumber\\
\displaystyle\leq \frac{2}{m+1}\left\lvert \int_0^T\int_\Omega \psi\cdot \nabla(u_{\delta_n}+\delta_n)^{\frac{m+1}{2}}((u_{\delta_n}+\delta_n)^{\frac{m-1}{2}}- u^{\frac{m-1}{2}})\ dx ds\right\rvert\nonumber\\
\displaystyle+ \frac{2}{m+1}\left\lvert \int_0^T\int_\Omega u^{\frac{m-1}{2}}\ \psi\cdot \left( \nabla(u_{\delta_n}+\delta_n)^{\frac{m+1}{2}}-\nabla u^{\frac{m+1}{2}}\right)\ dx ds\right\rvert\nonumber\\
\displaystyle\leq\frac{2}{m+1}||\psi||_4\ ||\nabla (u_{\delta_n}+\delta_n)^{\frac{m+1}{2}}||_2\ ||(u_{\delta_n}+\delta_n)^{\frac{m-1}{2}}-u^{\frac{m-1}{2}}||_4\nonumber\\
\displaystyle+  \frac{2}{m+1}\left\lvert \int_0^T\int_\Omega u^{\frac{m-1}{2}}\ \psi\cdot \left( \nabla(u_{\delta_n}+\delta_n)^{\frac{m+1}{2}}-\nabla u^{\frac{m+1}{2}}\right)\ dx ds\right\rvert.\nonumber
\end{array}
$
\\
Since $u^{\frac{m-1}{2}}\ \psi \in L^2((0,T)\times \Omega)$, we deduce from \eqref{za} and \eqref{zr} that the right-hand side of the above inequality converges to zero as $n\longrightarrow \infty$. In other words,
\begin{equation}
\label{zs}
\nabla (u_{\delta_n}+\delta_n)^m\rightharpoonup \nabla u^m\ \  \mathrm{in}\ L^{\frac{4}{3}}((0,T)\times \Omega), 
\end{equation}
for all $T>0$.\\

Now, we are going to show that $(u, \varphi)$ in Lemma \ref{conv} is the desired weak solution in Theorem \ref{Ex}. 
 Let $T>0$ and $\psi\in C^1([0,T]\times \overline{\Omega})$ with $\psi(T)=0$. The solution of \eqref{KSd} satisfies
\begin{equation}
\int_0^T\int_\Omega \left[ \nabla(u_{\delta_n}+\delta_n)^m\cdot \nabla\psi- u_{\delta_n}\ \nabla\varphi_{\delta_n}\cdot \nabla\psi-u_{\delta_n}\ \partial_t \psi\right]\ \mathrm{d}x \mathrm{d}t
=\int_\Omega u_0\ \psi(0, x)\ \mathrm{d}x,
\end{equation}
and,
\begin{equation}
\int_0^T\int_\Omega \left[\nabla \varphi_{\delta_n}\cdot \nabla \psi+M\ \psi-u_{\delta_n}\ \psi\right]\ \mathrm{d}x\mathrm{d}t=0.
\end{equation}

From \eqref{zs} we see that
$$\int_0^T\int_\Omega  \nabla(u_{\delta_n}+\delta_n)^m\cdot \nabla\psi\ \mathrm{d}x\mathrm{d}t\longrightarrow \int_0^T\int_\Omega  \nabla u^m \cdot \nabla\psi\ \mathrm{d}x\mathrm{d}t\ \mathrm{as}\ \delta_n\rightarrow 0.$$
From \eqref{za} we get
$$\int_0^T\int_\Omega u_{\delta_n}\ \partial_t \psi\ \mathrm{d}x \mathrm{d}t\longrightarrow \int_0^T\int_\Omega u\  \partial_t \psi\ \mathrm{d}x\mathrm{d}t \ \mathrm{as}\ \delta_n\rightarrow 0.$$
From \eqref{za} and \eqref{zaa} we get
$$\int_0^T\int_\Omega  u_{\delta_n}\ \nabla \varphi_{\delta_n}\cdot \nabla\psi\ \mathrm{d}x\mathrm{d}t\longrightarrow \int_0^T\int_\Omega u\ \nabla\varphi\cdot\nabla\psi \ \mathrm{d}x\mathrm{d}t\ \mathrm{as}\ \delta_n\rightarrow 0.$$
Thus we conclude that $u$ satisfies
$$\int_0^T\int_\Omega (\nabla u^m\cdot \nabla \psi-u\nabla \varphi\cdot\nabla\psi-u\cdot\partial_t \psi)\ \mathrm{d}x\mathrm{d}t=\int_\Omega u_0(x)\cdot \psi(0, x)\ \mathrm{d}x.$$
Similarly, from \eqref{zaa} we see that
$$\int_0^T\int_\Omega  \nabla\varphi_{\delta_n}\cdot \nabla\psi\ \mathrm{d}x\mathrm{d}t\longrightarrow \int_0^T\int_\Omega  \nabla\varphi\cdot \nabla\psi\ \mathrm{d}x\mathrm{d}t\ \mathrm{as}\ \delta_n\rightarrow 0,$$
and from \eqref{za} we see that
$$\int_0^T\int_\Omega u_{\delta_n}\ \psi\ \mathrm{d}x\mathrm{d}t\longrightarrow \int_0^T\int_\Omega u\  \psi\ \mathrm{d}x\mathrm{d}t\  \mathrm{as}\ \delta_n\rightarrow 0.$$
Thus, we have constructed a weak solution $(u,\varphi)$ of (KS).
\end{proof}
\subsection{Uniqueness}
In this section, we prove the uniqueness statement of Theorem \ref{Ex} under the additionnal assumption \eqref{assu} on $\varphi$. The proof relies on a classical duality technique, and on the method presented in \cite{local}

\begin{proof}
 The proof estimates the difference of weak solutions in dual space $H^{1}(\Omega)'$ of $H^1(\Omega)$, motivated by the fact that the nonlinear diffusion is monotone in this norm.\\
 
 Assume that we have two different weak solutions $(u_1, \varphi_1)$ and $(u_2, \varphi_2)$ to equations \eqref{PKS} corresponding to the same initial conditions, and fix $T>0$. We put
$$(u, \varphi)=(u_1-u_2, \varphi_1-\varphi_2) \ \mathrm{in}\ [0, T]\times \Omega.$$
Then $\varphi$ is the strong solution of
\begin{equation}\label{psi}
\begin{array}{cccr}
-\Delta \varphi&=&u&\  \mathrm{in}\  \Omega,\\
\partial_\nu \varphi&=& 0&\  \mathrm{on}\  \partial \Omega,\\
<\varphi>&=&0.
\end{array}\end{equation}
Since $\partial_t u \in L^2((0,T); H^1(\Omega)')$, we have 
$$-\Delta \partial_t\varphi= \partial_t u_1-\partial_t u_2=\partial_t u\ \ \mathrm{in}\ H^1(\Omega)',$$ 
and 
\begin{eqnarray}
\frac{1}{2}\frac{d}{dt}||\nabla \varphi||_2^2&=&\int_\Omega \nabla \varphi\cdot\nabla \partial_t\varphi\ dx\nonumber\\
&=& -<\Delta \partial_t \varphi, \varphi>_{(H^1)', H^1}=<\partial_t u, \varphi>_{(H^1)', H^1}.\label{norme}
\end{eqnarray}
Now it follows from \eqref{PKS} that $u$ satisfies the equation\\
\begin{equation}\label{fer}
\left\{
\begin{array}{lll}
\partial_t u&=&\mathrm{div}(\nabla(u_1^m-u_2^m))-\mathrm{div}(u_1 \nabla \varphi+u \nabla \varphi_2)\\
\partial_\nu u&=&0\\
u(0,x)&=&0.
\end{array}
\right.
\end{equation}
Substituting \eqref{fer} in \eqref{norme}, we obtain
\begin{eqnarray}
\frac{1}{2}\frac{d}{dt}||\nabla \varphi||_2^2&=&\int_\Omega (u_1^m-u_2^m)\ \Delta \varphi \ \mathrm{d}x+\int_\Omega u_1\ |\nabla\varphi|^2\mathrm{d}x
+\int_\Omega u\ \nabla\varphi_2\cdot \nabla \varphi\ \mathrm{d}x.\label{inte}
\end{eqnarray} 
The first integral on the right-hand side of \eqref{inte} is nonnegative due to the fact that $z\mapsto z^m$ is an increasing function. The second integral on the right-hand  side of \eqref{inte} can be estimated by  
$$\left\lvert\int_\Omega u_1\ |\nabla\varphi|^2\ \mathrm{d}x\right\rvert\leq ||u_1||_{\infty}\ \int_\Omega |\nabla \varphi|^2\ \mathrm{d}x.$$
For the last integral, using an integration by parts we obtain
\begin{eqnarray}
\label{limite}
\int_\Omega u\ \nabla\varphi_2\cdot \nabla \varphi\ \mathrm{d}x&=&-\int_\Omega \Delta \varphi\  \nabla\varphi_2\cdot \nabla \varphi\ dx\nonumber\\
&=& \int_\Omega \nabla \varphi\cdot\nabla(\nabla \varphi_2\cdot\nabla \varphi)\ dx\nonumber\\
&=&\sum_{i,j}\int_\Omega \partial_i\varphi\ \partial^2_{ij}\varphi_2\  \partial_j\varphi\ \mathrm{d}x+\sum_{i,j}\int_\Omega \partial_i\varphi\ \partial_j\varphi_2\ \partial^2_{ij}\varphi\ \mathrm{d}x.\end{eqnarray}
integrating by parts the second integral on the right-hand side of \eqref{limite},
\begin{eqnarray}
\sum_{i,j}\int_\Omega \partial_i\varphi\ \partial_j\varphi_2\ \partial^2_{ij}\varphi\ \mathrm{d}x&=& \sum_{i,j}\frac{1}{2}\int_\Omega \partial_j\varphi_2 \ \partial_{j} |\partial_i\varphi|^2\ \mathrm{d}x\nonumber\\
&=&\frac{1}{2}\int_\Omega \nabla \varphi_2\cdot\nabla(|\nabla\varphi|^2)\ \mathrm{d}x\nonumber\\
&=&-\frac{1}{2}\int_\Omega \Delta \varphi_2\ |\nabla\varphi|^2\ \mathrm{d}x\nonumber\\
&\leq& C(T) \ ||\nabla \varphi||_2^2,\nonumber
\end{eqnarray}
since $-\Delta \varphi_2=u_2-<u_2>\in L^\infty((0,T)\times\Omega)$. Together with \eqref{limite} the previous inequality implies
\begin{eqnarray}
\left\lvert\int_\Omega u\ \nabla\varphi_2\cdot \nabla \varphi\ \mathrm{d}x\right\rvert&\leq& C(T) \int_\Omega (|D^2\varphi_2|+1)\ |\nabla\varphi|^2\ \mathrm{d}x.\nonumber\\
&\leq& C(T)\ (||\varphi_2||_{L^\infty((0,T); W^{2,\infty}(\Omega))}+1)\ \ \int_\Omega |\nabla\varphi|^2\ dx,\nonumber
\end{eqnarray}
provided that the $L^\infty((0,T); W^{2,\infty}(\Omega))$ norm of the function $\varphi_2$ is bounded. 
 Thus, substituting the above estimates in \eqref{inte}, one finally obtains
\begin{equation}\label{estimate} \frac{d}{dt}\int_\Omega|\nabla \varphi|^2\ dx\leq C(T) \int_\Omega |\nabla \varphi|^2\ dx.\end{equation}
 Notice that $||\nabla \varphi(0)||_2=0$ which follows from \eqref{psi} and the property $u(0)=0$. Thus, inequality \eqref{estimate} implies
$$||\nabla \varphi(t)||^2_2\leq e^{C(T)\ t}\ ||\nabla \varphi(0)||_2^2=0.$$
Consequently, $\nabla \varphi(t)=0$ for all $t\in [0,T]$ and, since $<\varphi(t)>=0$, we have $\varphi(t)=0 $ for all $t\in [0,T]$. Using \eqref{psi}, we conclude that $u(t)=0$ for all $t\in [0,T]$. Consequently $(u_1, \varphi_1)=(u_2, \varphi_2)$.
\end{proof}
\section*{ACKNOWLEDGMENT}
I thank professors Philippe Lauren\c{c}ot and Marjolaine Puel for their helpful advices and comments during this work.


\begin{thebibliography}{30}
\bibitem{lpb}
N.~D.~Alikakos. $L^p$ bounds of solutions of      reaction-diffusion equations.
\emph{Communication in Partial Differential Equations 4(1979), no. 8, 827-868.}
\bibitem{local}
J.~Bedrossian, N.~Rodriguez and A.~Bertozzi. Local global well-possedness for aggregation equations and Patlak-Keller-Segel models with degenerate diffusion, \emph{Nonlinearity 24 (2011)1683-1714}.
\bibitem{globalexis}
J.~Bedrossian, I.~C.~ Kim. Global Existence and Finite Time Blow-Up for Critical Patlak-Keller-Segel Models with Inhomogeneous Diffusion, preprint arXiv:1108.5301.
\bibitem{onthepar}
A.~Blanchet. On the parabolic-elliptic {P}atlak-{K}eller-{S}egel system in
  dimension 2 and higher.
\emph{ To appear in S\'emin. \'Equ. D\'eriv. Partielles.}
\bibitem{critical}
A.~Blanchet. J.~A.~Carrillo. Ph.~Lauren{\c{c}}ot. Critical mass for a Patlak-Keller-Segel model with degenerate diffusion in higher dimensions. \emph{Calc.~Var.~Partial Differential Equations 35 (2009), no. 2,133-168}.
\bibitem{volume}
V.~Calvez, J.~A.~Carrillo. Volume effect in the Keller-Segel model: energy estimates preventing blow-up. \emph{J.~Math.~Pures.~Appl.~86 (2006)155-175}.
\bibitem{finitetim}
T.~ Cie{\'s}lak, Ph.~Lauren{\c{c}}ot. Finite time blow-up for radially symmetric solutions to a critical quasilinear Smoluchowski-Poisson system.\emph{C.R. Acad. Sci. Paris, ser. I 347(2009) 237-242}.
\bibitem{finite}
T.~Cie{\'s}lak, M.~Winkler. Finite-time blow-up in a quasilinear system of chemotaxis. \emph{Nonlinearity 21 (2008) 1057-1076}.
\bibitem{optimal}
J.~Dolbeault, B.~Perthame.  Optimal critical mass in the two-dimensional Keller-Segel model in $\mathbb{R}^2$. \emph{C. R. Math. Acad. Sci. Paris 339 (2004), no. 9, 611616}.
\bibitem{from}
D.~Horstmann. From 1970 until present: the {K}eller-{S}egel model in chemotaxis and its consequences. {I}.
\emph{ Jahresber. Deutsch. Math.-Verein. 105(3):103-165, 2003.}
\bibitem{liapunov}
D.~Horstmann. Lyapunov functions and $L^p$-estimates for a class of reaction-diffusion systems.  \emph{Colloq. math. 87 (2001) no. 1, 113-127}.
\bibitem{onexplo}
W.~J\"ager, S.~ Luckhaus, On explosions of solutions to a system of partial differential equations modelling chemotaxis.\emph{Trans. Amer. Math. Soc. 329(2)(1992) 819-824}.
\bibitem{initiation}
E.~F. Keller and L.~A. Segel. Initiation of slide mold aggregation viewed as an instability.
\emph{ J. Theor. Biol.}, 26:399--415, 1970.
\bibitem{preventing}
R.~Kowalczyk. Preventing blow-up in a chemotaxis model. \emph{J.~Math.~ Anal.~Appl. 305 (205)566-588}.
\bibitem{blowup}
T.~Nagai. Blow-up of nonradial solutions to parabolic-elliptic systems modelling chemotaxis in two-dimensional domains. \emph{J. Inequal. Appl. 6 (2001) 37-55}.
\bibitem{blow}
T.~Nagai. Blow-up of radially symmetric solutions to a chemotaxis system, \emph{Advances in Mathematical Sciences and Applications 5 (1995), no. 2, 581-601}.
\bibitem{random}
C.S.~ Patlak, Random walk with persistence and external bias, \emph{Bull. Math. Biophys. 15 (1953) 311-338.}
\bibitem{pde}
B.~ Perthame, PDE models for chemotactic movements. Parabolic, hyperbolic and kinetic, \emph{Appl. Math. 49 (6) (2005) 539-564.}
\bibitem{compact}
J.~Simon, Compact sets in the space $L^p(0,T; B)$. \emph{1987, Annali di Mathematica Pura ed Applicata (IV), vol. CXLVI, 65-96}.
\bibitem{global}
Y.~Sugiyama. Global existence in sub-critical cases and finite time blow-up in supercritical cases to degenerate Keller-Segel systems. \emph{Differential and Integral Equations 19 (2006), no. 8, 841-876.}
\bibitem{time}
Y.~Sugiyama. Time global existence and asymptotic behavior for solutions to degenerate quasi-linear parabolic systems of chemotaxis. \emph{Differential and Integral Equations 20 (2007), no. 2, 133-180.}
\end{thebibliography}
\end{document}